\documentclass[11pt]{amsart}

\pdfoutput=1
\usepackage{amsfonts, amstext, amsmath, amsthm, amscd, amssymb}
\usepackage{enumerate}
\usepackage{epsfig, xcolor, graphicx, overpic}
\usepackage{multirow}
\usepackage{pinlabel}

\let\oldmarginpar\marginpar
\renewcommand\marginpar[1]{\oldmarginpar[\raggedleft\footnotesize #1]%
{\raggedright\footnotesize #1}}

\usepackage[
pdfborderstyle={},
pdfborder={0 0 0},
pagebackref,
pdftex]{hyperref}

\renewcommand*{\backref}[1]{}
\renewcommand*{\backrefalt}[4]{
  \ifcase #1 %
   [No citations.]%
  \or
   [#2]%
  \else
   [#2]%
  \fi
}

\newtheorem{theorem}{Theorem}[section]
\newtheorem{lemma}[theorem]{Lemma}
\newtheorem{corollary}[theorem]{Corollary}

\newtheorem{proposition}[theorem]{Proposition}

\newtheorem*{claim}{Claim}

\newtheorem*{namedtheorem}{\theoremname}
\newcommand{\theoremname}{testing}

\theoremstyle{definition}
\newtheorem{remark}[theorem]{Remark}

\newtheorem{definition}[theorem]{Definition}

\newcommand{\refthm}[1]{Theorem~\ref{Thm:#1}}
\newcommand{\reflem}[1]{Lemma~\ref{Lem:#1}}
\newcommand{\refprop}[1]{Proposition~\ref{Prop:#1}}
\newcommand{\refcor}[1]{Corollary~\ref{Cor:#1}}

\newcommand{\refdef}[1]{Definition~\ref{Def:#1}}
\newcommand{\refsec}[1]{Section~\ref{Sec:#1}}
\newcommand{\reffig}[1]{Figure~\ref{Fig:#1}}
\newcommand{\reftab}[1]{Table~\ref{Tab:#1}}

\newcommand{\ZZ}{\mathbb{Z}}

\newcommand{\RR}{\mathbb{R}}

\newcommand{\FF}{\mathcal{F}}

\renewcommand{\setminus}{{\smallsetminus}}
\newcommand{\bdy}{\partial}

\def\OO{\mathcal{O}} 
\def\PP{\mathcal{P}}

\definecolor{light-gray}{gray}{0.6}
 
\newcommand{\OK}{\OO_K}

\newcommand{\Omin}{\OO_{\rm{min}}}

\newcommand{\quotient}[2]{{\raisebox{0.2em}{$#1$}
           \!\!\left/\raisebox{-0.2em}{\!$#2$}\right.}}

\def\vol{\mathrm{Vol}} 

\def\area{\mathrm{Area}}

\def\Isom{\mathrm{Isom}}

\def\area{\mathrm{Area}}
\def\min{\mathrm{min}}

\def\exp{\mathrm{exp}}  
\def\arccosh{\mathrm{arccosh}}

\def\HH{\mathbb{H}}

\def\FF{\mathcal{F}}

\begin{document}

\title{The lowest volume $3$--orbifolds with high torsion}
\author{Christopher K. Atkinson}
\address{Division of Science and Mathematics, 
University of Minnesota Morris,
Morris, MN 56267}
\email{catkinso@morris.umn.edu}

\author{David Futer}
\address{Department of Mathematics, Temple University,
Philadelphia, PA 19122}
\email{dfuter@temple.edu}
\thanks{Futer was supported in part by NSF grant DMS--1408682 and the Elinor Lunder Founders' Circle Membership at the Institute for Advanced Study.}
\date{\today}

\begin{abstract} 
For each natural number $n \geq 4$, we determine the unique lowest volume hyperbolic $3$--orbifold whose torsion orders are bounded below by $n$. This lowest volume orbifold has base space the $3$--sphere and singular locus the figure--$8$ knot, marked $n$. We apply this result to give sharp lower bounds on the volume of a hyperbolic manifold in terms of the order of elements in its symmetry group.
\end{abstract}

\maketitle

\section{introduction}

The volume of a hyperbolic manifold is a fundamental and very useful topological invariant. 
J{\o}rgensen  and Thurston showed that the
set of volumes of hyperbolic $3$--manifolds forms a closed, non-discrete, well--ordered subset
of $\RR^+$ \cite{thurston:notes}. 
Dunbar and Meyerhoff proved that the set of volumes of hyperbolic
$3$--orbifolds --- namely, quotients of manifolds by a discrete group action --- has the same properties \cite{dunbar-meyerhoff}.  It follows that any
family of hyperbolic manifolds or orbifolds contains a finite number
of members realizing the smallest volume.

In the past decade, volume minimizers have been found for several important families.   Gabai, Meyerhoff, and Milley \cite{gmm:smallest-cusped, milley} proved that the smallest volume orientable hyperbolic
$3$--manifold is the Weeks manifold $M_W$, with $\vol(M_W) \approx 0.9427$. Gehring, Marshall, and Martin  \cite{gehring-martin:minimal-orbifold, marshall-martin:minimal-orbifold2} have identified the smallest volume orientable $3$--orbifold $\Omin$, with $\vol(\Omin) \approx 0.03905$. There are also known results about non-compact manifolds and orbifolds \cite{adams:gieseking, adams:small-volume-orbifolds, cao-meyerhoff, meyerhoff:small-orbifold}, fibered manfiolds \cite{aaber-dunfield}, and several other families \cite{adams:limit-volume-orbifolds, agol:2cusped, atkinson-futer}.

In this paper, we identify the volume minimizers in infinitely many distinct families of orbifolds: namely, the family $L_n$ of hyperbolic $3$--orbifolds whose torsion orders are bounded below by $n$. For each $n \geq 4$, we prove that the unique smallest-volume orbifold in family $L_n$ is the figure--$8$ knot marked $n$, as in \reffig{examples}. See \refthm{main} below. In addition, we show that the volume of a hyperbolic $3$--manifold $M$ is bounded below by a constant times the size of its symmetry group $G$, where the constant depends on the orders of elements in $G$. 
These bounds are sharp in infinite families of examples.

\subsection{Orbifolds}
A $3$--dimensional, orientable \emph{orbifold} $\OO$ is locally modeled on $B^3/G$,
where $B^3$ is the closed $3$--ball and $G\subset SO(3)$ is a finite group
acting on $B^3$ by rotations.  The underlying topological space or \emph{base
space} is denoted $X_{\OO}$.  If a point in $B^3$ is fixed by a non-trivial
element of $G$, the quotient of this point in $X_{\OO}$ belongs to the
\emph{singular locus} $\Sigma_{\OO}$.

All manifolds and all orbifolds appearing in this paper are assumed to be orientable.
For a $3$--orbifold $\OO$, the base space $X_{\OO}$ is a
$3$--manifold, possibly with boundary. The singular locus $\Sigma_{\OO}$ is an 
embedded graph whose vertices in
the interior of $X_{\OO}$ have valence $3$, and whose vertices on $\bdy
X_{\OO}$ have valence $1$. A regular neighborhood of a point in the interior of an
edge of $\Sigma_{\OO}$ is the quotient of $B^3$ under $\ZZ_n$, a finite cyclic group
of rotations.  The integer $n\geq 2$ is called the \emph{torsion order}
of the edge. A regular neighborhood of a trivalent vertex $v \in \Sigma_{\OO}$ is the
quotient of $B^3$ under an (orientation preserving) spherical triangle group.

A $3$--orbifold is called \emph{hyperbolic} if $\OO = \HH^3 /\Gamma$, where $\Gamma$
is a discrete subgroup of $\Isom^+ (\HH^3)$, possibly with torsion.  In parallel to the case where $\OO$ is 
a manifold, the
group $\Gamma$ is called the \emph{orbifold fundamental group of $\OO$} and denoted
$\pi_1(\OO)$. For hyperbolic orbifolds, the singular locus $\Sigma_\OO$ is the quotient of fixed points of torsion elements of $\Gamma$.

\begin{definition}
For $n\geq2$, let $L_n$ denote the set of all orientable hyperbolic
$3$--orbifolds with non-empty singular locus and with all torsion orders
bounded below by $n$.  Thus 
\[ L_2 \supset L_3 \supset L_4 \supset \ldots
\quad \mbox{and} \quad \bigcap L_n = \emptyset. \]
\end{definition}

Being a member of $L_n$ imposes some additional structure. For instance,
an orbifold with torsion orders bounded below by $ n \geq 3$ is automatically orientable, because orientation-reversing isometries have order $2$. When $n \geq 4$ and $\OO \in L_n$, \reflem{LargeNLink} shows that the singular locus $\Sigma_\OO$ is a closed $1$--manifold, hence $\OO$ is a \emph{link orbifold} as in \refdef{link-orbifold}.

The main result of this paper identifies the volume minimizer in $L_n$ for each $n \geq 4$.

\begin{figure}
\begin{overpic}[width=1.4in]{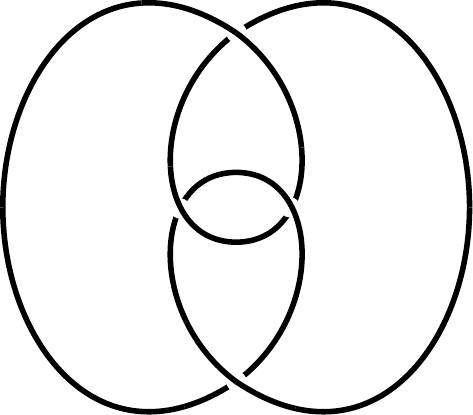}
\put(90,0){$n$}
\end{overpic}
\hspace{1in}
\begin{overpic}[width=1.6in]{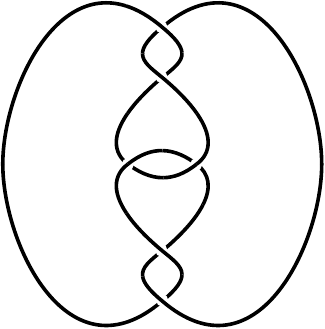}
\put(87,0){$3$}
\end{overpic}
  \caption{Left: the orbifold $\PP_n$, with singular locus the figure--$8$ knot and torsion label $n$. By \refthm{main}, $\PP_n$ is the unique volume minimizer among hyperbolic orbifolds with torsion orders bounded below by $n \geq 4$.
    Right: the orbifold $\OO_K$, with singular locus the $5_2$ knot and torsion label $3$.}
  \label{Fig:examples}
\end{figure}

\begin{theorem}\label{Thm:main}
For all $n\geq 4$, the unique lowest--volume element
of $L_n$ is the orbifold $\PP_n$,
with base space $S^3$ and singular locus the
figure--$8$ knot labeled $n$. See \reffig{examples}, left.
\end{theorem}

An exact formula for the volume of $\PP_n$ was computed by Vesnin and Mednykh \cite[Lemma 1]{vesnin-mednykh}.
See \reftab{MomBound} below for a table of approximate values. As $n \to \infty$,  $\vol(\PP_n)$ converges to $2V_3 = 2.02988\dots$, the volume of the figure--$8$ knot complement. See \refsec{calculus} for quantitative estimates on the rate of convergence.

The hypothesis that $n\geq 4$ is crucial for \refthm{main}, since $\PP_2$ is a spherical orbifold and $\PP_3$ is Euclidean  \cite[Example 2.33]{chk:orbifold}. Here is the state of knowledge about volume minimizers in $L_n$ for these small values of $n$.

By the work of Gehring, Marshall, and Martin \cite{gehring-martin:minimal-orbifold, marshall-martin:minimal-orbifold2},  the unique volume minimizer of $L_2$ is the minimal--volume orbifold $\OO_\min$, of volume $\vol(\OO_\min) = 0.03905 \dots$.

The volume minimizer in $L_3$ is currently unknown.  In \refcor{3torsion}, we build on our previous work \cite[Theorem 1.6]{atkinson-futer}
to show that $\vol(\OO) \geq 0.2371$ for every $\OO \in L_3$.
We conjecture that the unique volume minimizer in $L_3$ is  the orbifold
$\OK$ depicted in \reffig{examples}, right. Since $\vol(\OK) =
\vol(M_W)/3 = 
0.3142 \dots$,   there is currently a $25\%$ gap between the conjectured and proven volume bounds.

\subsection{Symmetry groups}
One major reason to search for the smallest volume orbifold with a particular property is that any finite group $G$ acting on a manifold $M$ by isometries has an orbifold as a quotient. Thus universal lower bounds on the volume of $M/G$ lead to upper bounds on the size of $G$. This connection is already present in dimension $2$, where Hurwitz's famous $84(g-1)$ theorem about the maximal symmetry group of a genus $g$ hyperbolic surface follows as a corollary of identifying $\pi/21$ as the minimal area of a hyperbolic $2$--orbifold \cite{hurwitz:84g, macbeath:hurwitz-theorem}. In a similar vein,  Gehring, Marshall, and Martin's identification of the smallest--volume orientable $3$--orbifold 
$\Omin$ implies that a hyperbolic $3$--manifold $M$ with symmetry group $G$ must have volume at least $ \vol (\Omin) |G| > 0.039 |G|$.

As a consequence of \refthm{main} and some intermediate results, we obtain a refinement of their result that depends on the orders of elements in $G$. Note that when $|G|$ is odd, the constant multiplying $|G|$ is much larger than $0.039$.

\begin{theorem}\label{Thm:volume-special-case}
Let $M$ be an orientable hyperbolic $3$--manifold of finite volume, and let $G$ be a group of orientation-preserving isometries of $M$.
 Let $p$ be the smallest prime number dividing $|G|$, or else $1$ if $G$ is trivial.
 Then
\begin{equation*}
\vol(M) \: \geq \: |G| \cdot w_p,
\quad
\mbox{where}
\quad
w_p = \begin{cases}
\vol(\OO_\min) \geq 0.03905 & p=2 \\
0.2371 & p=3 \\
\vol(\PP_5) \geq 0.9372 & p=5 \\
\vol(M_W) \geq 0.9427 & \mbox{otherwise}.
\end{cases}
\end{equation*}
\end{theorem}

See \refthm{closed-mfld-volume} in \refsec{symmetry} for a slightly sharper version of \refthm{volume-special-case}, which takes into account the way in which $G$ acts. See also \refthm{cusped-mfld-volume} for stronger volume estimates under the additional hypothesis that $M$ has cusps. That result strengthens and generalizes the work of Adams  \cite{adams:limit-volume-orbifolds} and Futer, Kalfagianni, and Purcell \cite{fkp:conway} on cusped manifolds with symmetry.

\subsection{Organization}
We begin the proof of \refthm{main} by restricting the topology of $\OO$. In \refsec{link}, we show that the singular locus of any volume minimizer $\OO \in L_n$ must be a link, with all components labeled $n$. In \refsec{moms}, we show that this link has a large embedded tube. When this tube is drilled out, the work of Agol and Dunfield \cite{ast} combined with that of Gabai, Meyerhoff, and Milley \cite{gmm:smallest-cusped} implies  that $\OO \setminus \Sigma_\OO$ is one of only ten cusped hyperbolic manifolds. In other words, $\OO$ must be obtained by \emph{Dehn filling} one of these ten manifolds. See \refthm{few-parents}.

In \refsec{computer}, we prove \refthm{main} for $4 \leq n \leq 14$. We use a theorem of Futer, Kalfagianni, and Purcell \cite{fkp:volume} to limit the possibilities for $\OO$ to finitely many Dehn fillings of the manifolds from \refthm{few-parents}. Then, we rigorously check the finitely many fillings by computer, following Milley \cite{milley}, to finish the proof.

In \refsec{calculus}, we finish the proof of \refthm{main} for $n \geq 15$. For these values of $n$, \refthm{few-parents} tells us that $\OO$ is a Dehn filling of one of just \emph{two} surgery parents, both of which have volume $2V_3 \approx 2.03$. Thus it suffices to estimate the change in volume during Dehn filling. We do this using the work of Hodgson and Kerckhoff on cone-manifolds \cite{hodgson-kerckhoff}.

Finally, in \refsec{symmetry}, we apply the above results to relate the volume of a hyperbolic manifold to the size of its symmetry group.

\section{Link orbifolds}\label{Sec:link}

The goal of this section is to relate small-volume hyperbolic orbifolds to \emph{link orbifolds}, which were the main object of study in our previous paper \cite{atkinson-futer}.

\begin{definition}\label{Def:link-orbifold}
A $3$--orbifold $\OO$ is called a \emph{link orbifold} if its singular locus $\Sigma_\OO$
is a closed $1$--manifold (i.e., a link) in $X_{\OO}$.
\end{definition}

We will show, in \refprop{reduction}, that for any $n \geq 3$, a volume minimizer $\OO \in L_n$ must be a link orbifold with $n$--torsion only. Combined with \cite[Theorem 1.6]{atkinson-futer}, this will imply lower bounds on the volume of any orbifold $\OO \in L_3$. In the next section, we use these results to further restrict the topology of any volume minimizer $\OO \in L_n$.

\begin{lemma}\label{Lem:LargeNLink}
Let $\OO$ be a hyperbolic $3$--orbifold with no $2$--torsion. Then $\Sigma_\OO$ has no vertices. Furthermore, either $\OO$ is a link orbifold, or it is non-compact, with a $S^2(3,3,3)$ cusp cross-section.
\end{lemma}

\begin{proof}
Suppose, for a contradiction, $\Sigma_{\OO}$ contains a vertex $v$. Then the boundary of a regular neighborhood of $v$ is a spherical $2$--orbifold $S^2(p,q,r)$, where $
\frac{1}{p} + \frac{1}{q} + \frac{1}{r} > 1$.
This cannot occur in our setting, where all torsion orders of $\OO$ are bounded below by $3$.

Next, suppose that a singular edge $e \subset \Sigma_\OO$ is not a closed geodesic. Since $\OO$ has no vertices, $e$ must be a bi-infinite geodesic with both endpoints at a cusp. The cross-section of such a cusp is a Euclidean $2$--orbifold, either $S^2(2,2,2,2)$ or $S^2(p,q,r)$ where $\frac{1}{p} + \frac{1}{q} + \frac{1}{r} = 1$. The only possibility with no $2$--torsion is $p=q=r=3$.
\end{proof}

\begin{lemma}\label{Lem:turnover-volume}
Let $\OO$ be a hyperbolic $3$--orbifold with no $2$--torsion, and suppose $\OO$ is not a link orbifold. Then $\OO$ has $3$--torsion, and $\vol(\OO) \geq V_3 / 2 =  0.50747 \dots$, where $V_3$ is the volume of a regular ideal tetrahedron in $\HH^3$.
\end{lemma}

\begin{proof}
By \reflem{LargeNLink}, $\OO$ is non-compact and has a  $S^2(3,3,3)$ cusp cross-section. Orbifolds of this type were analyzed by Adams~\cite[Section 4]{adams:small-volume-orbifolds}. He showed that there are exactly three hyperbolic $3$--orbifolds with a $S^2(3,3,3)$ cusp cross-section and volume less than $V_3 / 2$. We call these exceptional orbifolds $\OO_1, \OO_2, \OO_3$.

For each $\OO_i$, let $C_i$ be a horospherical neighborhood of a $S^2(3,3,3)$ cusp, which is \emph{maximal} in the sense that no larger neighborhood is embedded. Lifting $C_i$ to the universal cover $\HH^3$ produces an arrangement of horoballs tangent at their boundaries, called a \emph{horoball packing}. For each $\OO_i$, Adams gives a complete description of the horoball packing of $\HH^3$; see  \cite[Figure 3]{adams:small-volume-orbifolds}. In particular, these three orbifolds share the feature that a fundamental domain for $\bdy C_i \cong S^2(3,3,3)$ has a unique point of self-tangency. As Adams notes \cite[Lemma 2.1]{adams:small-volume-orbifolds}, this unique point of self-tangency of $\bdy C_i$ lifts to a point of tangency between two horoballs in $\HH^3$, with an order--$2$ elliptic element of $\pi_1(\OO_i)$ interchanging the two horoballs.

This cannot occur under the hypotheses of this lemma, as our orbifold $\OO$ has no $2$--torsion. Therefore, $\vol(\OO) \geq V_3 / 2$.
\end{proof}

\begin{lemma}\label{Lem:reduction}
Fix $n \geq 3$, and let $\OO \in L_n$ be a hyperbolic link orbifold. If $n = 3$, suppose in addition that $\vol(\OO) < 0.4408$.
Let $\OO'_n$ be a link orbifold with the same base space and singular locus as $\OO$, but with all torsion orders changed to $n$. 

Then $\OO'_n$ is hyperbolic and $\vol(\OO'_n) \leq \vol(\OO)$, with equality if and only if $\OO'_n = \OO$.
\end{lemma}

\begin{proof}
This is a restatement of \cite[Proposition 2.8 and Lemma 5.1]{atkinson-futer}. The hard part of the proof is showing that $\OO'_n$ is hyperbolic; once this is established, the volume inequality follows from Schl\"afli's formula \cite[Theorem 3.20]{chk:orbifold}.
\end{proof}

Combining the last two results gives us a lot of control over volume minimizers in $L_n$.

\begin{proposition}\label{Prop:reduction}
Let $n \geq 3$, and let $\OO \in L_n$ be a hyperbolic orbifold that minimizes volume among all orbifolds in $L_n$. Then
\begin{enumerate}
\item\label{Itm:link} $\OO$ is a link orbifold.
\item\label{Itm:same} Every component of $\Sigma_{\OO}$ has torsion order exactly $n$.
\end{enumerate}
\end{proposition}

\begin{proof}
If $n \geq 4$, conclusion \eqref{Itm:link} is a restatement of \reflem{LargeNLink}. If $n=3$, and $\OO \in L_3$ is not a link orbifold, \reflem{turnover-volume} says that $\vol(\OO) > 0.5074$. This volume is not minimal in $L_3$, because there is an orbifold $\OO_K \in L_3$, as in \reffig{examples} (right), such that $\vol(\OO_K) \approx 0.31423$. Thus a volume minimizer  $\OO \in L_n$ must be a link orbifold, proving \eqref{Itm:link}.

Conclusion  \eqref{Itm:same} is immediate from \reflem{reduction}, because the orbifold $\OK$ of \reffig{examples} shows that a volume minimizer in $L_3$ must have volume less than $0.44$.
\end{proof}

\begin{corollary}\label{Cor:3torsion}
Let $\OO \in L_3$ be a hyperbolic orbifold. Then $\vol(\OO) \geq 0.2371$.
\end{corollary}

\begin{proof}
By \refprop{reduction}, it suffices to consider the case where $\OO \in L_3$ is a link orbifold. Now, \cite[Theorem 1.6]{atkinson-futer} gives the estimate $\vol(\OO) \geq 0.2371$.
\end{proof}

\section{Topological reduction}\label{Sec:moms}

The goal of this section is to drastically restrict the topological and
combinatorial possibilities for the volume minimizer in $L_n$ for $n \geq 4$. Given any
orbifold $\OO$ that minimizes volume in $L_n$, we will show in
\refthm{few-parents} that the
complement of the singular locus is one of ten topological types.  This reduces the problem of finding the
volume minimizer in $L_n$ to a manageable Dehn surgery problem, which will be
analyzed in Sections~\ref{Sec:computer} and \ref{Sec:calculus}.

The proof of \refthm{few-parents} proceeds in several steps.
\begin{enumerate}[1.]

\item By \refprop{reduction}, we already know that the singular locus $\Sigma_\OO$ is a link with all components marked $n$.

\smallskip

\item For any link orbifold $\OO \in L_n$, there is a large embedded tube about $\Sigma_\OO$. Explicit estimates on the tube radius (in terms of $n$) follow from results of Gehring, Maclachlan, Martin, and Reid \cite{gmmr, gehring-martin:commutators-collars}, and are derived in \refprop{tube}.

\smallskip

\item We drill out the singular locus of $\OO$, obtaining a cusped hyperbolic
  $3$--manifold $M_\OO = X_\OO \setminus \Sigma_\OO$. Given lower bounds on
  the radius of a tube about $\Sigma_\OO$, a theorem of Agol and Dunfield   \cite{ast}
  gives lower bounds on the ratio $\vol(\OO)/\vol(M_\OO)$.
This bound, in
  terms of $n$, is derived in \refprop{drillbound}.

\smallskip

\item Gabai, Meyerhoff, and Milley have enumerated the ten cusped hyperbolic $3$--manifolds of lowest volume \cite{gmm:smallest-cusped}. Combining their enumeration with \refprop{drillbound}  and the knowledge that $\vol(\OO) \geq \vol(\PP_n)$ completes the proof of \refthm{few-parents}.
\end{enumerate}

The first step of this program requires studying embedded tubes about the singular locus.

\begin{definition}\label{Def:xn}
For $n \geq 4$, define a sequence
  \[x_n =
	\begin{cases}
	  \quotient{1}{2} + \quotient{\sqrt{3}}{2} = 1.3660...
	  &\text{if } n =4\\ 
	   1.84028
	  &\text{if } n =5\\ 
	   2.41383
	  &\text{if } n =6\\ 
	   2.65579
	  &\text{if } n =7\\ 
	  -1+ \quotient{1}{2 \sin^2(\pi/n) }
		&\text{if } n\geq 8.
	\end{cases}
	\]
\end{definition}

\begin{proposition}\label{Prop:tube}
Let $\OO$ be a finite--volume, orientable, hyperbolic  link orbifold, all of whose torsion orders are equal to $n \geq 4$. Then the singular locus $\Sigma_\OO \subset \OO$ has an embedded tubular neighborhood of radius at least 
\[
r_n = \arccosh (x_n) / 2,
\]
where $x_n$ is as in \refdef{xn}.
\end{proposition}

\begin{proof}
This is a consequence of the work of Gehring, Maclachlan, Martin, and Reid \cite{gmmr, gehring-martin:commutators-collars} on axial distances in Kleinian groups. Let us describe the relevant results.

First, let $\OO = \HH^3 / \Gamma$ be any (orientable) hyperbolic $3$--orbifold. A singular geodesic $\alpha \subset \Sigma_\OO$ with torsion label $n \geq 2$ lifts to a $\Gamma$--orbit of geodesic axes in $\HH^3$, with each axis fixed pointwise by an elliptic subgroup $\ZZ_n \subset \Gamma$. When $n \geq 4$, Gehring and Martin \cite{gehring-martin:commutators-collars} showed that two distinct axes of $n$--torsion must be separated by distance at least $b(n)$, where
\[
\cosh b(n)  \: = \: 
	\begin{cases}
	   1 + \quotient{\sqrt{5}}{5} 
	  &\text{if } n =5\\ 
	   2
	  &\text{if } n =6\\ 
	  -1+ \quotient{1}{2 \sin^2(\pi/7) }
		&\text{if } n= 7 \\
	x_n & \text{if } n= 4 \text{ or } n \geq 8 .
	\end{cases}
\]
The definition of $b(n)$ appears in \cite[equations (4.16) and (4.17)]{gehring-martin:commutators-collars}, and the lower bound on axial distance is proved in Theorem 4.18 of that paper. See also \cite[Section 4]{gmm:axial-distances}.

If all torsion subgroups of $\Gamma$ are of the same order $n \geq 4$, and the elliptic axes are separated by distance at least $b(n)$, then tubes of radius $b(n)/2$ about these axes will project to disjointly embedded tubes about $\Sigma_\OO$ in $\OO = \HH^3 / \Gamma$. This proves the proposition for $n \neq 5,6,7$.

To complete the proof for these values of $n$, we need to relate orbifold groups to $2$--generator Kleinian groups. Note that a geodesic $\alpha \subset \Sigma_\OO$ is called \emph{simple} if its lift to $\HH^3$ is a disjoint union of hyperbolic lines, with disjoint endpoints at infinity. In a link orbifold, all singular geodesics are simple. The following statement is  \cite[Lemma A.3]{atkinson-futer}; see also \cite[Lemma 2.26]{gehring-martin:commutators-collars}.

\begin{lemma}\label{Lem:gavens-lemma}
Let $\OO = \HH^3 / \Gamma$ be a hyperbolic link orbifold, all of whose torsion orders are equal to $n \geq 3$. Then there is a Kleinian group $G = \langle f,g \rangle$, generated by an elliptic $f \in \Gamma$ of order $n$ and an elliptic $g$ of order $2$, such that the axis of $f$ is again simple in $G$, and such that the distance $\delta(f,g)$ between the axes of $f$ and $g$ is equal to the radius of a maximal tube about $\Sigma_\OO$. Furthermore, there is an index--$2$ subgroup $H \subset G$ that is also a subgroup of $\Gamma$.
\end{lemma}

 Gehring, Maclachlan, Martin, and Reid have gone a long way toward classifying such $2$--generator groups $G = \langle f,g \rangle$ for which the distance $\delta(f,g)$ is particularly small \cite{gmmr}.

When $n=5$, they show \cite[Table 3]{gmmr} that if $\delta(f,g) < 0.6097$, then $G$ must be conjugate to one of three groups $G_{5,i}$ for $i=1,2,3$. By \cite[Table 12]{gmmr}, these three groups have the property that the axis of $f$ is not simple. Thus, by \reflem{gavens-lemma}, these three groups cannot occur in our setting, when $\OO$ is a link orbifold. Hence, the radius of an embedded tube about $\Sigma_\OO$ is at least
\[ \delta(f,g) \geq 0.6097 \geq \arccosh(1.84028) / 2.
\]

When $n=6$, a similar argument applies. By \cite[Table 4]{gmmr}, we see that if $\delta(f,g) < 0.7642$, then $G$ must be the group $G_{6,1}$ that realizes the minimal axial distance for $n=6$. However, by \cite[Table 12]{gmmr},  the axis of $f$ is not simple in this group. Thus, under our hypotheses, the radius of an embedded tube about $\Sigma_\OO$ is at least
\[ \delta(f,g) \geq 0.7642 \geq \arccosh(2.41383) / 2.
\]

When $n=7$, \cite[Table 5]{gmmr} shows that if $\delta(f,g) < 0.8162$, then $G$ must be the group $G_{7,1}$ that realizes the minimal axial distance for $n=7$. For this group, the distance $\delta(f,g)$ between axes $f$ and $g$ is exactly $b(7)/2$. As Gehring and Martin describe in \cite[Example 8.11]{gehring-martin:commutators-collars}, this group $G_{7,1}$ is the $(2,3,7)$ triangle group, which contains $3$--torsion in addition to $7$--torsion. But then the index--$2$ subgroup $H \subset G$ would also contain $3$--torsion, which is a contradiction because $H \subset \Gamma$ and all torsion orders of $\Gamma$ are $7$. Thus, under our hypotheses, the radius of an embedded tube about $\Sigma_\OO$ is at least
\[ \delta(f,g) \geq 0.8162 \geq \arccosh(2.65579) / 2,
\]
completing the proof.
\end{proof}

\begin{remark}
A modified version of the above argument applies for $n=3$, with more special cases to consider. See \cite[Theorem A.1]{atkinson-futer} for the best available statement.
\end{remark}

The argument now proceeds by considering the cusped manifold $M_{\OO}$
obtained by drilling $\Sigma_{\OO}$ out from $\OO$. We use the following result 
to estimate the change in volume when drilling.

\begin{proposition}\label{Prop:drillbound}
Let $\OO$ be a finite--volume, orientable, hyperbolic  link orbifold, all of whose torsion orders are equal to $n \geq 4$.
Let $M_{\OO} = X_{\OO} \setminus \Sigma_{\OO}$ be the cusped hyperbolic manifold
obtained by drilling out $\Sigma_\OO$. Then 
\[
\frac{\vol(\OO)}{\vol(M_{\OO})} \: \geq \:   \frac{( x_n^2 -1)^{3/2}}{x_n^3} 
\left( 1 + \frac{0.91}{x_n} \right)^{-1},
\]
where $x_n$ is as in \refdef{xn}.
\end{proposition}

\begin{proof}
\refprop{tube} gives a lower bound $r_n$ on the radius of an embedded tubular neighborhood of $\Sigma_\OO$, where 
\[x_n = \cosh  2 r_n \, . \]
Now, a theorem of Agol and Dunfield (see \cite[Theorem 10.1]{ast} and \cite[Page 2299]{agol-culler-shalen}) implies
\[
\frac{\vol(M_{\OO})}{\vol(\OO)} \: \leq \:  (\coth^3{2r_n})
\left( 1 + \frac{0.91}{\cosh{2r_n}} \right) \: = \: \left(  \frac{x_n}{\sqrt{ x_n^2 -1}} \right)^{\! 3}
\left( 1 + \frac{0.91}{x_n} \right).
\]
See  \cite[Proposition 4.3]{atkinson-futer} for a proof of the above estimate in the setting of link orbifolds. Taking reciprocals gives the estimate in the statement of the proposition.
\end{proof}

\begin{theorem}[Gabai--Meyerhoff--Milley \cite{gmm:smallest-cusped}]\label{Thm:moms}
Let $N$ be a cusped orientable hyperbolic $3$--manifold. If $\vol(N) \leq 2.848$, then $N$ has exactly one cusp, and is one of the Snappea census manifolds $\tt m003$,
$\tt m004$, $\tt m006$, $\tt m007$, $\tt m009$, $\tt m010$, $\tt m011$, $\tt m015$, $\tt m016$, or $\tt m017$.

Furthermore, if $\vol(N) \leq 2.568$, then $N = \tt m003$ or
$\tt m004$, and $\vol(N) = 2V_3 = 2.02988...$
\end{theorem}

\begin{proof}
The enumeration of cusped manifolds of volume at most $2.848$ is  \cite[Corollary 1.2]{gmm:smallest-cusped}. The two volume minimizers on the list, $\tt m003$ and $\tt m004$, have volume $2V_3$ and were previously identified by Cao and Meyerhoff \cite{cao-meyerhoff}. The next smallest volume is $\vol({\tt m006}) \geq 2.5689$. 
\end{proof}

We can now prove the main result of this section.

\begin{table}
\begin{tabular}{| c | c | c | c | c |}
\hline
$n$  & \parbox{0.19\textwidth}{Lower bound if $\vol(M_\OO) \geq 2.568$} & \parbox{0.19\textwidth}{Lower bound if \\ $\vol(M_\OO) \geq 2.848$} & \parbox{0.20\textwidth}{Lower bound if $\vol(M_\OO) \geq 3.6638$} & $\vol(\PP_n)$ \\
\hline
4 & \textcolor{light-gray}{0.48729} & 0.54043 & 0.69524 & 0.50747 \\
5 & 1.01654 & \textcolor{light-gray}{1.12738} & 1.45034 & 0.93720 \\ 
6 & 1.40605 & \textcolor{light-gray}{1.55936} & 2.00606 & 1.22128   \\
7 & 1.52065 & \textcolor{light-gray}{1.68646} & 2.16958 & 1.41175   \\
8 & \textcolor{light-gray}{1.40625} & 1.55958 & \textcolor{light-gray}{2.00635} & 1.54386  \\
9 &  1.72961 & \textcolor{light-gray}{1.92420} & \textcolor{light-gray}{2.47542} & 1.63860 \\
10 & 1.93362 & \textcolor{light-gray}{2.15116} & \textcolor{light-gray}{2.76740} & 1.70857  \\
11 & 2.06917 & \textcolor{light-gray}{2.30195} & \textcolor{light-gray}{2.96139} & 1.76158  \\
$\geq 11$ &  $\geq 2.06917$ & \textcolor{light-gray}{$\geq 2.30195$} & \textcolor{light-gray}{$\geq 2.96139$} & $ < 2.02990 $ \\ 
\hline
\end{tabular}
\vspace{1ex}

\caption{\refprop{drillbound} gives lower bounds on the volume
of an orbifold $\OO\in L_n$, if one assumes a lower bound on the volume of
the cusped manifold $M_\OO = \OO \setminus \Sigma_\OO$. All values are
truncated at 5 decimal places. Values listed in gray appear for completeness
only, and are not used in our arguments.} \label{Tab:MomBound}
\end{table}

\begin{theorem}\label{Thm:few-parents}
Suppose that $n \geq 4$, and let $\OO \in L_n$ be a volume minimizer in $L_n$.
Then $M_\OO = \OO \setminus \Sigma_\OO$ is one of the Snappea census manifolds $\tt m003$,
$\tt m004$, $\tt m006$, $\tt m007$, $\tt m009$, $\tt m010$, $\tt m011$, $\tt m015$, $\tt m016$, or $\tt m017$.

Furthermore, if $n \neq 4,8$, then $M_\OO = \OO \setminus \Sigma_\OO$ is one of $\tt m003$ or $\tt m004$.
\end{theorem}

\begin{proof}
By \refprop{reduction}, $\OO$ is a link orbifold, such that each component of $\Sigma_\OO$ has torsion order exactly $n$. Now, \refprop{drillbound} gives a lower bound on the ratio  $\vol(\OO)/\vol(M_\OO)$, and \refthm{moms} gives a lower bound on $\vol(M_\OO)$. Multiplying these together gives a lower bound on $\vol(\OO)$; see \reftab{MomBound}.

When $n =4$ or $n = 8$, we look for lower bounds  for $\vol(\OO)$ under the assumption that $\vol(M_\OO) \geq 2.848$. As \reftab{MomBound} shows, the lower bound for $\vol(\OO)$ exceeds the volume of $\PP_n$. Thus a volume minimizer in $L_n$ must come from one of the ten census manifolds enumerated in \refthm{moms}.

When $n \geq 5$ and $n \neq 8$, we obtain sufficiently good lower bounds for $\vol(\OO)$ under the weaker assumption that $\vol(M_\OO) \geq 2.568$. As \reftab{MomBound} shows, the lower bound for $\vol(\OO)$ exceeds the volume of $\PP_n$ when $5 \leq n \leq 11$ (except for $n=8$). In fact, for $n = 11$, the lower bound for $\vol(\OO)$ exceeds the volume of the figure--8 knot complement $\tt m004$, which is greater than $\vol(\PP_n)$ for every $n$. For $n > 11$, the same lower bound on $\vol(\OO)$ applies, since by \reflem{reduction} the volume can only go down when we change the torsion labels of $\OO$ to $11$.
Thus, if $\OO$ is a volume minimizer in $L_n$ for $n \geq 5$ and $n \neq 8$,
\refthm{moms} tells us that $M_\OO$ is one of $\tt m003$ or $\tt
m004$.
\end{proof}

We also obtain the following statement, which will be useful in \refsec{symmetry}.

\begin{theorem}\label{Thm:multicusped-drill}
Suppose that $n \geq 4$, and let $\OO \in L_n$ be a link orbifold that is either non-compact or has disconnected singular locus.
Then the volume of $\OO$ is bounded below by the value in the next-to-last column of \reftab{MomBound}.
\end{theorem}

\begin{proof}
We may assume without loss of generality that all torsion orders of $\OO$ are equal to $n$. (Otherwise, replace $\OO$ by  $\OO'_n$ as in \reflem{reduction}, reducing volume in the process.)

Let $M_\OO$ be the cusped $3$--manifold obtained by drilling out $\Sigma_\OO$. Under our hypotheses, $M_\OO$ has at least two cusps, hence the work of Agol \cite{agol:2cusped} implies that $\vol(M_\OO) \geq V_8 = 3.6638 \dots$, where $V_8$ is the volume of a regular ideal octahedron. Now, we may apply \refprop{drillbound}, as in the last proof, and obtain the volume estimates in the next-to-last column of \reftab{MomBound}.
\end{proof}

 \section{The case $n \leq 14$: Computer analysis of Dehn fillings}\label{Sec:computer}

Following \refthm{few-parents}, we know that for a volume minimizer $\OO \in L_n$, the complement of the singular locus is one of ten cusped hyperbolic manifolds. In other words, $\OO$ is obtained by \emph{orbifold Dehn filling} one of these ten cusped manifolds. Let us review what this means.

\begin{definition}
Suppose that $M$ is a hyperbolic $3$--manifold with a torus cusp
$C$.  A \emph{slope} on $C$ is an unoriented, non--trivial homology class on
$\partial C$. If we endow $H_1(\partial C) \cong \ZZ^2$ with a basis $\langle
\mu, \lambda \rangle$, then we may write $s \in H_1(\partial C)$ as
$\pm(a\mu+b \lambda)$.  If $a$ and $b$ are relatively prime, then the Dehn
filling $M(s)$ is the manifold resulting from attaching a non--singular
solid torus to $M$, with the meridian disk mapped to $s$.  If $(a,b) = n(a',b')$,  
where $n > 1$ and $(a',b')$ is a primitive element of $\ZZ^2$,
then $M(s)$ is a link orbifold obtained by attaching a singular
solid torus to $M$, with the meridian disk mapped to $a' \mu + b' \lambda$, and with the core curve 
carrying the torsion label
$n=\gcd(a,b)$. 
\end{definition}

To prove \refthm{main}, it remains to find \emph{which} Dehn filling of the cusped manifolds of \refthm{few-parents} has the smallest volume. For any given value of $n$, we can restrict this to a finite search via the following theorem of 
Futer, Kalfagianni,
and Purcell~\cite[Theorem 1.1]{fkp:volume}.  We state it here for filling along a single cusp; see  \cite[Theorem 4.5]{atkinson-futer} for a general statement in the context of orbifolds with multiple cusps.

\begin{theorem}[Futer--Kalfagianni--Purcell \cite{fkp:volume}]\label{Thm:fkp}
Let $M$ be a complete, finite--volume hyperbolic $3$--manifold. Let $C$ be an embedded horospherical cusp neighborhood in $M$, and let $s$ be a slope on $\bdy C$. If the length of $s$ is $\ell(s) > 2\pi$, the Dehn filling $M(s)$ is hyperbolic, and
$$\vol(M(s)) \geq \left( 1- \left(
\frac{2\pi}{\ell(s)}\right)^2 \right)^{3/2} \vol(M).$$
\end{theorem}

We can now prove \refthm{main} in the case of $4 \leq n \leq 14$.

 \begin{theorem}\label{Thm:n-less14}
  Let  $\OO$ be a volume minimizer in $L_n$, where $4 \leq n \leq 14$.
  Then $\OO =  \PP_n$.
 \end{theorem}

 \begin{proof} 

		\refthm{few-parents} implies that $\OO$ is a Dehn filling of $M$, where $M$ belongs to the set
	  \[\mathcal{M} =
	  \{\tt{m003}, \, \tt{m004}, \, \tt{m006}, \, \tt{m007}, \, \tt{m009}, \,
	  \tt{m010},\, \tt{m011}, \, \tt{m015}, \, \tt{m016}, \, \tt{m017}\} .\] 
Fix a value $B_n$ slightly larger than $\vol(\PP_n)$. Since $\OO$ is a volume minimizer, we know that $\vol(\OO) < B_n$.
Let $M\in \mathcal{M}$.  Let $s =
	  a\mu+b\lambda$ be a slope on $\partial M$. Since we are looking for slopes that produce orbifolds with $n$--torsion, we want 
	   $\gcd(a,b) = n$. Meanwhile, the constraint $\vol(\OO) < B_n$ places an upper bound on the length $\ell(a\mu + b\lambda)$.
By \refthm{fkp}, any filling slope $s$ that gives a volume minimizer must be chosen from  the set 
\[ \FF_M^n(B_n) := \left\{ (a,b) \in \ZZ^2 \:  : \: 1- \left(
		\frac{2\pi}{\ell(a\mu + b \lambda)}\right)^2 \leq \left(
		\frac{B_n}{\vol(M)} \right)^{2/3}
		\text{ and } \gcd(a,b) = n
		 \right\} ,\]
where $M \in \mathcal{M}$.
For each $n$, this gives a finite set of manifold--slope pairs.

For each $n$, this finite set of manifold--slope pairs can be rigorously
analyzed by computer
 using a combination of Snap \cite{snap-tube} and code developed by Moser \cite{moser:thesis} and Milley \cite{milley}. We can use Snap to compute approximate hyperbolic structures on the Dehn fillings in $\FF_M^n(B_n)$, and the Milley--Moser code to certify that their volumes are all higher than that of $\PP_n$ (except for $\PP_n$ itself). See Milley \cite{milley} for a detailed description of the method.

For $n=4$, the computer analysis appears in our previous work \cite[Section 4]{atkinson-futer}. We start by analyzing the case $n=8$.
Milley's code rigorously certifies that $\vol(\PP_8) < 1.55$, hence we may take $B_8 = 1.55$. The set
\[\bigcup_{M \in \mathcal{M}} \FF_M^8(1.55)\]
		consists of $42$ slopes.  For each of these fillings, Snap finds an
		approximate solution to the gluing equations in which all tetrahedra
		are positively oriented. For each such approximate solution, we use
		Milley's implementation of Moser's algorithm to rigorously verify
		that there is an actual positively oriented solution 
		(i.e., an actual hyperbolic structure)
		near the
		approximate solution.  This algorithm also gives an upper bound on
		the distance between each approximate solution and the actual
		solution.  We then use Milley's code to rigorously verify that
		each filled orbifold has volume at least $1.55$.  The
		results of this verification show that indeed, each orbifold under
		consideration except $\PP_8$ has volume greater than $1.55$, proving
		the theorem in the case where $n = 8$.  The code is available in the
		ancillary files \cite{af-data}.

For all other values of $n$, \refthm{few-parents} implies that it suffices to let $M \in  \{\tt{m003}, \, \tt{m004} \}$.
We let
$B_n$ equal the value of $\vol(\PP_n)$ rounded up to the
 nearest one-hundredth (see \reftab{MomBound}). Then
 \[
 \FF_{\mathtt{m003}}^n(B_n) = \{ (-2n, n), (-n, n), (-n,2n), (0,n), (n,0),
 (n,n)\}\]
 and
 \[\FF_{\mathtt{m004}}^n(B_n) = \{ (-2n, n), (-n, n), (0,n), (n,0), (n,n), (2n,
 n)\} . 
 \]

 For each of these Dehn fillings, we use Snap~\cite{snap-tube} to find a positively
 oriented approximate solution to the gluing equations, and Milley's code to 
 rigorously compute the volume.
 For each filling
 $\OO$ except the orbifolds $\PP_n$ and $\mathtt{m003}(-10,10)$, the
 Milley--Moser algorithm certifies that $\vol(\OO) > \vol(\PP_n)$.  See the
 ancillary files \cite{af-data}.
 
 For the initial  triangulation of
 $\mathtt{m003}(-10,10)$ given by Snap, Moser's upper bound on distance
 between the approximate solution and the actual hyperbolic structure is too large for
 Milley's algorithm to certify that the volume is larger than $B_{10}$.
 Retriangulating and running the algorithm again certifies that its volume is
 also larger than $B_{10}$.  See the files in the subfolder {\tt bad-ten} of the
 ancillary files.
 \end{proof}

The method of the above proof is likely to work for $n$ in any bounded range, although the example of $\mathtt{m003}(-10,10)$ shows that the computer may need occasional hand-holding. 
However, to prove \refthm{main} in complete generality, we need to consider an infinite set of Dehn fillings of $\tt{m003}$ and $\tt{m004}$. This requires a different method.

\section{The case $n\geq 15$: Analytic estimates on the change in volume}\label{Sec:calculus}

In this section, we complete the proof of \refthm{main} by showing that the orbifold $\PP_n$ is the unique volume minimizer in the family $L_n$, for all $n \geq 4$. Following \refthm{n-less14}, it suffices to consider $n \geq 15$.
\refthm{few-parents} tells us that for $n \geq 9$, we only need to consider Dehn fillings of $\tt{m003}$ and $\tt{m004}$. 
Because $\vol({\tt m003}) = \vol({\tt m004}) = 2V_3$, it suffices to consider the \emph{change} in volume when we perform Dehn filling.

\begin{definition}\label{Def:vol-change}
Let $M$ be a cusped hyperbolic $3$--manifold with a horospherical cusp
neighborhood $C$. For a slope $s$ on $\bdy C$,  we define an \emph{inverse
norm} 
\[
A(s) = \frac{(2\pi)^2 \area(\bdy C)}{ \ell(s)^2 } ,
\]
where  $\ell(s)$ is the length of a geodesic representative of $s$ on $\bdy C$. 
Since expanding or shrinking the cusp $C$ rescales the Euclidean metric on $\bdy C$, it will leave $A(s)$ invariant. Note  that $A(s) = (2\pi/L(s))^2$, where $L(s)$ is the normalized length of Hodgson and Kerckhoff \cite{hodgson-kerckhoff}.

For any Dehn filling slope $s$ yielding a hyperbolic $3$--manifold or orbifold $M(s)$, define
\[
\Delta V(s) = \vol(M) - \vol(M(s)).
\]
\end{definition}

The last remaining step in the proof of \refthm{main} is the following result.

\begin{proposition}\label{Prop:vol-change}
Fix $n \geq 15$. Then, among all Dehn fillings of $\mathtt{m003}$ or $\mathtt{m004}$ that produce
link orbifolds with $n$--torsion, the  largest value of $\Delta V$ is
uniquely realized by $\mathtt{m004}(n,0)$, the filling that  yields
$\PP_n$.
\end{proposition}

Gromov and Thurston showed that $\Delta V(s)$ is always strictly positive  \cite[Theorem 6.5.6]{thurston:notes}. Neumann and Zagier \cite{neumann-zagier} showed that $\Delta V(s)$ is well-predicted by the inverse norm of a slope:
\[
\Delta V(s) = \frac{A(s) }{4} + O(A(s)^2).
\]
However, since they do not quantify the error term $O(A(s)^2)$, their theorem is primarily an asymptotic estimate. 
Instead, we will prove \refprop{vol-change} by applying the work of Hodgson
and Kerckhoff on deformations of cone manifolds \cite{hodgson-kerckhoff}.
Their estimate on the change in volume $\Delta V(s)$ is in terms of the
following functions:

\[
l(z) = \frac{ 3.3957 \, z^2 (z^2-3)  (z^4 + 4z^2-1)}{2 \, (z^4-6 z^2+1)(z^2 +
1)^2}  \, , \]
\[
u(z) = \frac{ 3.3957 \, z^2 (z^4 + 4z^2-1)}{2 \, (z^2 + 1)^3}  \, ,  \]
  \[
  f(z) = \frac{ 3.3957(1-z) } { \exp\left(\int_1^z F(w)\,dw\right) } 
 \quad \mbox{where} \quad
 F(z) = - \frac{z^4 +6z^2 +4z +1}{(z+1)(z^2 + 1)^2} \, ,\]
\[  \tilde{f}(z) = \frac{3.3957(1-z)} {\exp\left(\int_1^z \tilde{F}(w)\,dw\right)}
 \quad \mbox{where} \quad
   \tilde{F}(z) = -
	 \frac{z^6+7z^4+12z^3-9z^2-4z+1}{(z+1)(z^2+1)(z^2+2z-1)(z^2-2z-1)} \, .\]


\begin{theorem}[Hodgson--Kerckhoff \cite{hodgson-kerckhoff}] \label{Thm:HK-estimate}
Let $M$ a cusped hyperbolic $3$--manifold. Let $s$ be a slope on some cusp of $M$, satisfying
 $A(s) \leq f(1/\sqrt{3}) = 0.68653 \dots$.
  Then $M(s)$ is hyperbolic and
 \[ \int_{\tilde{z}}^{1} l(z) \,dz
   \leq \Delta V(s) \leq
 \int_{\hat{z}}^{1} u(z) \,dz \, ,
   \]
   where $\hat{z} = f^{-1} (A(s))$ and $\tilde{z} = \tilde{f}^{-1} (A(s))$.
 \end{theorem}

\begin{proof}
This is \cite[Theorem 5.12]{hodgson-kerckhoff}, with simplified notation. 
The integrands in their statement of the theorem are expressed in terms of functions $G(z)$, $\tilde{G}(z)$, and $H(z)$ that are defined on \cite[Page 1079]{hodgson-kerckhoff}. Those integrands simplify to $l(z)$ and $u(z)$, respectively.
\end{proof}

Implicit in the statement of \refthm{HK-estimate} is the claim that $f(z)$ and $\tilde{f}(z)$ are invertible in the right range. We verify this claim now.

\begin{lemma}\label{Lem:monotonic}
The functions $f(z)$ and $\tilde{f}(z)$ are strictly decreasing on the interval $J = [z_0,1]$, where $z_0 = \sqrt{-2 + \sqrt{5}} = 0.48586 \ldots $. The inverse functions $f^{-1}(x)$ and $\tilde{f}^{-1}(x)$ are well-defined for $x \in [0, 0.699]$, and map the interval $[0, 0.699]$ into $[z_0,1]$.
\end{lemma}

\begin{proof}
We compute that
\[f'(z) = -2 \cdot 3.3957
 \, \frac{z(z^4+4z^2-1)}{(z+1)(z^2+1)^2} \, \, \exp \left( \int_z^1 F(w)dw \right)  .
  \]
  All the factors in the rational function are uniformly positive when $z>0$, except for 
\[
p(z) = z^4+4z^2-1,
\]
which has a zero at $z_0 = \sqrt{-2+\sqrt{5}}$ and is positive thereafter. Thus $f(z) < 0$ when $z \in (z_0,1]$, hence $f$ is monotonically decreasing and invertible on the closure of that domain. This means that  $f^{-1}(x)$
  is well-defined and decreasing on the interval 
\[
  \left[ f(1), \, f( z_0 ) \right] =
	[0, \, 0.69910 \dots]. 
	\]
	
The analysis for $\tilde{f}(z)$ is similar. We compute that
\[\tilde{f}'(z) = -2 \cdot 3.3957 \, 	\frac{z (z^2 - 3)(z^4 + 4 z^2 - 1)}{(z+1)(z^2+1)(z^2+2z-1)(z^2-2z-1)}
\,\, \exp\left(\int_z^1 \tilde{F}(w)\,dw  \right)
.\]
Among the factors of the rational function, the only ones that have a zero on $(0,1]$ are $(z^2+2z-1)$, with a zero at $\sqrt{2}-1 < z_0$, and $p(z)$ with a zero at $z_0$. Thus $\tilde{f}(z) < 0$ when $z \in (z_0,1]$, hence  $\tilde{f}^{-1}(x)$
  is well-defined and decreasing on the interval 
\[
  \left[ \tilde{f}(1), \, \tilde{f}( z_0 ) \right] =
	[0, \, 1.6368 \dots]. \qedhere
	\]
\end{proof}

\begin{definition}\label{Def:slopes}
For a fixed $n$, let $\tilde{s} = \tilde{s}_n$ be the Dehn filling slope
$(n,0)$ on $\mathtt{m004}$ that produces the orbifold $\PP_n$. Let $s = s_n \neq \tilde{s}$ denote any other slope on a cusp of $\mathtt{m003}$ or $\mathtt{m004}$ that is $n$ times a primitive slope. The subscript $n$ will usually be implicit. 

In this notation, \refprop{vol-change} says that when $n \geq 15$,
\[
\Delta V(\tilde{s}) > \Delta V(s).
\]
\end{definition}

We will derive \refprop{vol-change} from \refthm{HK-estimate}, as follows.
\begin{enumerate}[1.]

\item We want to show that the lower bound on $\Delta V(\tilde{s})$ is larger than the upper bound on $\Delta V(s)$ for any potential competitor $s$. To that end, we show in \reflem{integrands} that the integrand $l(z)$ for the lower bound is
larger than the integrand $u(z)$ for the upper bound.

\smallskip

\item We will show in \reflem{integral} that the interval of integration for the lower bound contains the interval of integration for the upper bound. To do this, we will need

\begin{enumerate}
\item  estimates on $A(s)$ and $A(\tilde{s})$, achieved in \reflem{slopes}, and
\item an inequality relating the functions $f$ and $\tilde{f}$, achieved in
  \reflem{4f}.
\end{enumerate}
\end{enumerate}

\begin{lemma}\label{Lem:integrands}
Let $z_0 \approx 0.48586$ be as in \reflem{monotonic}. Then $0 \leq u(z) \leq l(z)$ on the interval $J = [z_0, 1]$, with strict inequalities in the interior.
\end{lemma}

\begin{proof}
Since 
\[u(z) = 
  \frac{ 3.3957 \, (z^4+4z^2-1)}{2 \, (1+z^2)^3},\]
we have $u(z) > 0$ if and only if $p(z) = z^4 + 4z^2 -1 > 0$.  As we have seen in the proof of \reflem{monotonic},
we have $p(z) > 0$ whenever $z > z_0$.  Thus
$u(z)> 0$ on the interior of $J$.

Next, observe that
\[\frac{l(z)}{u(z)} 
  =  \frac{z^4-2z^2 - 3 }{z^4 - 6z^2 + 1 }, \]
where the denominator $z^4 - 6z^2 + 1$ is negative on $J$.
Therefore, 
\[
  \frac{l(z)}{u(z)} > 1
\quad \Longleftrightarrow \quad
z^4-2z^2 - 3 < z^4 - 6z^2 + 1
\quad \Longleftrightarrow \quad
4 z^2  < 4,
\]
which holds on the interior of $J$.
\end{proof}

\begin{lemma}\label{Lem:slopes}
Let $s$ and $\tilde{s}$ be as in \refdef{slopes}. Then 
\[
A(\tilde{s}) = 8\sqrt{3}\pi^2/n^2 
\quad \mbox{and} \quad
A(s) \leq 2\sqrt{3}\pi^2/n^2.
\]
In particular, $4A(s) \leq A(\tilde{s})$.
\end{lemma}

\begin{proof}
  This can be easily computed using Snap \cite{snap-tube}.  In both manifolds  $\mathtt{m003}$ and $\mathtt{m004}$ , the maximal cusp $C$ has $\area(\bdy C) = 2 \sqrt{3}$. 
  
  For $\mathtt{m004}$, the shortest simple closed curve is $(1,0)$, the knot-theoretic meridian, of length $1$. Thus the slope $\tilde{s} = (n,0)$ has length $n$, hence by \refdef{vol-change}, $A(\tilde{s}) = 8\sqrt{3}\pi^2/n^2$.
  
The shortest primitive slope on a maximal cusp of $\mathtt{m003}$ has length $2$, while the second-shortest primitive slope on a maximal cusp of $\mathtt{m004}$ is the $(0,1)$ slope of length   $2 \sqrt{3}$. Thus, if $s$ is $n$ times a primitive slope and $s \neq \tilde{s}$, we have $\ell(s) \geq 2n$. Therefore, by \refdef{vol-change}, $A(s) \leq 2\sqrt{3}\pi^2/n^2$.
%
  \end{proof}

\begin{lemma}\label{Lem:interval}
Let $n \geq 15$. Then both   $\tilde{s} = \tilde{s}_n$ and $s = s_n$ satisfy the
hypotheses of \refthm{HK-estimate}. In addition, both $\tilde{z} =
\tilde{f}^{-1}(A(\tilde{s}))$ and $\hat{z} = f^{-1}(A(s))$ lie in the
interval $I = [0.83, 1]$.  
\end{lemma}
%

\begin{proof}
By \reflem{slopes}, we know that when $n \geq 15$, the inverse norms of $s$ and $\tilde{s}$ satisfy
\[ 
0 < 4 A(s) \leq A(\tilde{s}) \leq \frac{8\sqrt{3}\pi^2}{15^2} \approx 0.60781.
\]
This means that both $s$ and $\tilde{s}$ satisfy hypotheses of \refthm{HK-estimate}. By \reflem{monotonic}, it also means that the values $\hat{z} = f^{-1} (A(s))$ and $\tilde{z} = \tilde{f}^{-1} (A(\tilde{s}))$ are well-defined. Since $f^{-1}$ and $\tilde{f}^{-1}$ are monotonically decreasing by \reflem{monotonic}, we know that
\[
\tilde{z} = \tilde{f}^{-1}(A(\tilde{s})) \geq \tilde{f}^{-1} \left( \frac{8\sqrt{3}\pi^2}{15^2} \right) = 0.83211 \dots
\]
and
\[\hat{z} = f^{-1}(A(s)) \geq f^{-1} \left( \frac{2\sqrt{3}\pi^2}{15^2} \right) = 0.95182 \dots . \qedhere
\]
\end{proof}

\begin{lemma}\label{Lem:4f}
On the interval $I = [0.83,1]$, we have $\tilde{f}(z) \leq 4 f(z)$, with a strict inequality in the interior of $I$.
\end{lemma}

\begin{proof}
Note that $\tilde{f}(1) = 4 f(1) = 0$, while
\[
\tilde{f}(0.83) = 0.61577 \ldots < 1.75736 \ldots = 4f(0.83).
\]
Thus the desired inequality is true at the endpoints of $I$. To prove that $\tilde{f}(z) < 4 f(z)$ in the interior of $I$, we show that $4f$ is concave down on $I$ (hence above its secant line), while $\tilde{f}$ is concave up (hence below its secant line).

In \reflem{monotonic}, we computed $f'$ and $\tilde{f}'$. Now, we differentiate again:
 \[f''(z) = - 2 \cdot 3.3957 \, \frac{ z^8+6 z^6+32 z^4+10 z^2-1}{(z^2+1)^4 (z+1)}
\, \exp\left(\int_z^1 F(w)\,dw  \right) .
\]
All the factors in the denominator are positive. The polynomial $g(z) = z^8+6 z^6+32 z^4+10 z^2-1$ is increasing on $I$, and satisfies $g(0.83)>20$. Hence $g(z)>0$ on $I$, and, we conclude that
$f''(z)<0$ on $I$.

Next, we compute that
	\[\tilde{f}''(z) = -2\cdot 3.3957 \, 	 
	 \frac{z^{12}-4z^{10}+17z^8 -248z^6 +203 z^4 - 36z^2 +3}
	 {(z+1)(z^2+1)^2 (z^2+2z-1)^2 (z^2-2z-1)^2}
	\, \exp\left(\int_z^1 \tilde{F}(w)\,dw\right) .
	\]  
Since all factors in the denominator are positive, it follows that $\tilde f''(z) >0$ on
	$I$ if
	and only if 
\[
q(z) = z^{12}-4z^{10}+17z^8 -248z^6 +203 z^4 - 36z^2
	  +3 <0
\]
on $I$.  We check this by analyzing  $r(z) = q(\sqrt{z})$.

\begin{claim}
The function $r(z) = z^6 - 4z^5 + 17z^4 -248z^3 +203z^2- 36z +3$ is strictly negative on the interval $[0.83^2, 1]$. 
\end{claim}

We prove this claim by  ``downward induction'' on the derivatives of $r$.
In fact, we will show that 
 that for $k \in \{1,2,3,4\}$, $r^{(k)}(z)$
	  has no zeros on the larger interval $[0.5,1]$.  
	  
First, check that  $r^{(4)}$, a quadratic function, is always positive on
	  $[0.5,1]$.  This implies that $r^{(3)}$ is increasing on
	  $[0.5,1]$.  It follows from the fact that $r^{(3)}(1)<0$ that
	  $r^{(3)}$ is negative on the interval $[0.5,1]$, hence $r''$ is decreasing on $[0.5,1]$.  Since $r''(0.5)
	  < 0$, $r''$ is negative on $[0.5,1]$, hence $r'$ is decreasing.  Since $r'(0.5)<0$,  $r'$ is
	  negative on $[0.5,1]$.   Therefore, $r$ is decreasing on $[0.5,1]$. Since $r(0.83^2) \approx -3.226 < 0$, this proves the claim. 
	  
\smallskip
	  
The claim shows that $\tilde{f}''(z) > 0$ on $I$. Since we have already seen that $4f''(z) < 0$ on $I$, the proof is complete.
\end{proof}

\begin{lemma}\label{Lem:integral}
Let $n \geq 15$, and let $s$ and $\tilde{s}$ be slopes as in \refdef{slopes}. Let
$ \tilde{z} = \tilde{f}^{-1} (A(\tilde{s}))$ and $\hat{z} = f^{-1}(A(s))$. Then $\tilde{z} < \hat{z}$.
\end{lemma}

\begin{proof}
When $n \geq 15$, \reflem{interval} says that $\tilde{z}$ and $\hat{z}$ lie in the interval $I = [0.83, 1]$. 
Therefore,
\[
4 A(s) \leq A(\tilde{s}) = \tilde{f}(\tilde{z}) < 4 f ( \tilde{z} ),
\]
where the first inequality is by \reflem{slopes} and the second inequality is by \reflem{4f}.
Since $f^{-1}$ is a decreasing function and $A(s) < f(\tilde{z})$, we conclude that
\[
\hat{z} = f^{-1}(A(s)) > \tilde{z}. \qedhere
\]
\end{proof}

Finally, we can prove \refprop{vol-change}, the main result of this
section.

\begin{proof}[Proof of \refprop{vol-change}]
As in \refdef{slopes}, let $\tilde{s}$ be the slope on a cusp of ${\tt m004}$ that produces the orbifold $\PP_n$, and let $s \neq \tilde{s}$ be any other slope on  ${\tt m003}$ or ${\tt m004}$ that is $n$ times a primitive slope.
Now,
  \[
 \Delta V(s) \leq   \int_{\hat{z}}^1 u(z)\,dz
 <  \int_{\tilde{z}}^1 u(z) \,dz
 <  \int_{\tilde{z}}^1 l(z) \,dz
\leq \Delta V(\tilde{s}) ,
 \] 
where the first and last inequalities are \refthm{HK-estimate}, the second inequality is \reflem{integral}, and the third inequality is \reflem{integrands}.
\end{proof}

We can also complete the proof of \refthm{main}.

\begin{proof}[Proof of  \refthm{main}]
Let $\OO$ be a volume minimizer in $L_n$, where $n \geq 4$.
If $4 \leq n \leq 14$, \refthm{n-less14} proves that $\OO = \PP_n$. If $n \geq 15$, \refthm{few-parents} shows that $\OO$ is a Dehn filling of  ${\tt m003}$ or ${\tt m004}$. Now, \refprop{vol-change} shows that among those Dehn fillings that produce $n$--torsion, $\PP_n$ is the unique one of lowest volume.
\end{proof}

\section{Symmetry groups of manifolds}\label{Sec:symmetry}

In this section, we apply the prior results of this paper (\refthm{main},
\refcor{3torsion}, and \refthm{multicusped-drill}) to relate the volume of a
hyperbolic manifold to the size of its symmetry group. Recall, from the
introduction,
that bounds of this type were one of the main motivations driving the search for the lowest volume orbifold, compact and non-compact   \cite{gehring-martin:minimal-orbifold, marshall-martin:minimal-orbifold2, meyerhoff:small-orbifold}. 

\begin{definition}\label{Def:group-np}
Let $G$ be a group of isometries of a finite-volume hyperbolic $3$--manifold $M$. It is a well-known consequence of Mostow rigidity that $G$ is finite \cite[Corollary 5.7.4]{thurston:notes}. We define the following numerical invariants of $G$:
\begin{itemize}
\item $n=n(G)$ is the smallest order of a non-trivial element of $G$ stabilizing a point of $M$, or else $1$ if $G$ acts freely.
\item $p=p(G)$ is the smallest prime number dividing $|G|$, or else $1$ if $G$ is trivial.
\end{itemize}
We emphasize that $n(G)$ depends on how $G$ acts on $M$, whereas $p(G)$ depends only on $|G|$.
\end{definition}

The following theorem is a generalization of \refthm{volume-special-case}, stated in the introduction.

\begin{theorem}\label{Thm:closed-mfld-volume}
Let $M$ be an orientable hyperbolic $3$--manifold of finite volume, and let $G$ be a group of orientation-preserving isometries of $M$.
 Let $n$ and $p$ be as in \refdef{group-np}. 
 Then
\begin{equation*}
\vol(M) \: \geq \: |G| \cdot w_n
\: \geq \: |G| \cdot w_p,
\quad
\mbox{where}
\quad
w_i = \begin{cases}
\vol(\OO_\min) \geq 0.03905 & i=2 \\
0.2371 & i=3 \\
\vol(\PP_4) \geq 0.5074 & i=4 \\
\vol(\PP_5) \geq 0.9372 & i=5 \\
\vol(M_W) \geq 0.9427 & \mbox{otherwise}.
\end{cases}
\end{equation*}
Furthermore, the above estimate is sharp for infinitely many pairs $(M,G)$.
\end{theorem}

\begin{proof}
We consider two cases, depending on whether $G$ acts freely or with fixed points. 

If  $G$ acts freely, then $n=1$. The quotient is $M/G = N$, an orientable hyperbolic $3$--manifold. Thus, by the work of Gabai, Meyerhoff, and Milley \cite{gmm:smallest-cusped, milley}, 
\[
\vol(M) \: = \: |G| \cdot \vol(N) \: \geq \: |G| \cdot \vol(M_W)
\: = \: |G| \cdot w_1
\: \geq \: |G| \cdot w_p
\]
for any $p$. 

Otherwise, if the action of $G$ has fixed points, we have $n \geq p \geq 2$, hence $w_n \geq w_p$. The quotient is an orbifold $M / G = \OO \in L_n$. Combining \refthm{main}, \refcor{3torsion}, and the work of Gehring, Marshall, and Martin \cite{gehring-martin:minimal-orbifold, marshall-martin:minimal-orbifold2} gives
\[
\vol(M) = |G|\cdot  \vol(\OO) \geq 
\begin{cases}
|G| \cdot \vol(\OO_\min) & n=2 \\
|G| \cdot0.2371 & n=3 \\
|G| \cdot \vol(\PP_n) & \mbox{otherwise}.
\end{cases}
\]
Since $\vol(\PP_n) > \vol(M_W)$ for $n > 5$, we conclude that 
\[
\vol(M) \geq |G| \cdot w_n  \geq |G| \cdot w_p.
\]

To see the sharpness of this estimate, let $X$ be one of $\Omin$, $\PP_4$, $\PP_5$, or $M_W$.
Since Kleinian groups are residually finite \cite{alperin:selberg-lemma},
there are infinitely many epimorphisms $\varphi: \pi_1 X \to G$, with finite image and torsion-free kernel. For each such $\varphi$, we can let $M$ be the cover of $X$ corresponding to $\ker \varphi$. The theorem will then be sharp for $(M,G)$.
\end{proof}

There are similar (and sharper) lower bounds if one assumes that the manifold $M$ has cusps. Here, Meyerhoff \cite{meyerhoff:small-orbifold} and Adams \cite{adams:small-volume-orbifolds} have identified the first several non-compact $3$--orbifolds of lowest volume, and Adams has also identified the smallest volume $3$--orbifolds that could occur as the quotient of a one-cusped manifold \cite{adams:limit-volume-orbifolds}.
Futer, Kalfagianni, and Purcell have found lower bounds on the volume of a symmetric link complement in $S^3$, provided the symmetry group is cyclic \cite{fkp:conway}. The following result strengthens and generalizes those estimates.

\begin{theorem}\label{Thm:cusped-mfld-volume}
Let $M$ be a cusped orientable hyperbolic $3$--manifold, and let $G$ be any group of orientation-preserving isometries of $M$.
 Let $n$ and $p$ be as in \refdef{group-np}.  Then
\begin{equation*}
\vol(M) \: \geq \: |G| \cdot u_n
\: \geq \: |G| \cdot u_p,
\quad
\mbox{where}
\quad
u_i = \begin{cases}
V_3 / 12 \geq 0.08457 & i=2 \\
V_3 / 2 \geq 0.50747 & i=3 \\
0.69524 & i=4 \\
1.45034 & i=5 \\
2.00606 & i=6 \\
2 V_3 \geq 2.02988 & \mbox{otherwise}.
\end{cases}
\end{equation*}
Furthermore, if $M$ has a single cusp, then $\vol(M) \geq |G| \cdot 0.30532\dots,$ where the decimal in the equation is $1/12$ the volume of a regular ideal octahedron.
\end{theorem}

\begin{proof}
As in the proof of \refthm{closed-mfld-volume}, we consider two cases, depending on how $G$ acts.

If $G$ acts freely, then $n=1$, hence $u_1 \geq u_p$ for any $p$. The quotient $N = M/G$ is a cusped orientable hyperbolic manifold. By the work of Cao and Meyerhoff \cite{cao-meyerhoff}, $\vol(N) \geq 2V_3$.

If $G$ acts with fixed points, then $n \geq p \geq 2$, hence $u_n \geq u_p$. The quotient $M/G$ is a non-compact hyperbolic orbifold $\OO \in L_n$,
whose volume can be estimated as follows. If $n=2$, then $\vol(\OO) \geq V_3/12$ by the work of 
 Meyerhoff \cite{meyerhoff:small-orbifold} and Adams \cite{adams:small-volume-orbifolds}. If $n=3$, then a cusp cross-section of $\OO$ is either a torus or $S^2(3,3,3)$; in either case we have the estimate $\vol(\OO) \geq V_3 / 2$, by
 \cite[Lemma 2.2]{atkinson-futer}  and \reflem{turnover-volume} of this paper.  If $4 \leq n \leq 6$, we turn to \refthm{multicusped-drill} and obtain volume estimates from \reftab{MomBound}. 
 
When $n \geq 7$, any orbifold $\OO \in L_n$ is also an element of $L_7$. By \reftab{MomBound}, the volume estimate for a non-compact orbifold $\OO \in L_7$ is larger than $2.1 > 2V_3 \approx 2.03$. Therefore, for $n \geq 7$, the smaller volume bound comes from the case where $M/G$ is a manifold.

The ``furthermore'' statement is a theorem of Adams \cite[Corollary 8.3]{adams:limit-volume-orbifolds}.
\end{proof}

%
%
%
%

\bibliographystyle{hamsplain}
\bibliography{./atkinson-futer_references}
\end{document}